\documentclass[reqno,12pt]{amsart}

\usepackage{microtype}
\usepackage{amsfonts, amsmath, amsthm, amstext, amssymb, amscd, color, rotating}
\usepackage{url}
\usepackage{paralist}
\usepackage{tikz}
\usepackage[english]{babel}
\usepackage{mathtools}
\usepackage{framed}
\usepackage{amsrefs}

\usepackage{hyperref} 
\usepackage{enumerate}

\newtheorem{lemma}{Lemma}[section]
\newtheorem{theorem}[lemma]{Theorem}
\newtheorem{proposition}[lemma]{Proposition}
\newtheorem{corollary}[lemma]{Corollary}

\begin{document}

\title[Squeezing function on doubly-connected domains via LDE]{The squeezing function on doubly-connected domains via the Loewner differential equation}

%    author one information
\author{Tuen Wai Ng }
\address{The University of Hong Kong, Pokfulam, Hong Kong}
\email{ntw@maths.hku.hk}

%    author two information
\author{Chiu Chak Tang}
\address{The University of Hong Kong, Pokfulam, Hong Kong }
\email{ChiuChakTang@connect.hku.hk}

%    author three information
\author{Jonathan Tsai}
\address{The University of Hong Kong, Pokfulam, Hong Kong}
\email{jonathan.tsai@cantab.net}

\date{\today}



\maketitle

\begin{abstract}
For any bounded domains $\Omega$ in $\mathbb{C}^{n}$, Deng, Guan and Zhang introduced the squeezing function $S_\Omega (z)$ which is a biholomorphic invariant of bounded domains. We show that for $n=1$,  the squeezing function on an annulus $A_r = \lbrace z \in \mathbb{C} : r <|z| <1 \rbrace$ is given by $S_{A_r}(z)= \max \left\lbrace |z| ,\frac{r}{|z|} \right\rbrace$ for all $0<r<1$. This disproves the conjectured formula for the squeezing function proposed by Deng, Guan and Zhang and establishes (up to biholomorphisms) the squeezing function for all doubly-connected domains in $\mathbb{C}$ other than the punctured plane. It provides the first non-trivial formula for the squeezing function for a wide class of plane domains and answers a question of Wold. Our main tools used to prove this result are the Schottky-Klein prime function  (following the work of Crowdy) and a version of the Loewner differential equation on annuli due to Komatu. We also show that these results can be used to obtain lower bounds on the squeezing function for certain product domains in $\mathbb{C}^{n}$.\end{abstract}

%==========================
%  Introduction
%==========================

\section{Introduction}
\label{sect:intro}

In 2012, Deng, Guan and Zhang \cite{squ_def1} introduced the squeezing function of a bounded domain $\Omega$ in $\mathbb{C}^n$ as follows. For any $z \in \Omega$, let $\mathcal{F}_{\Omega}(z)$ be the collection of all embeddings $f$ from $\Omega$ to $\mathbb{C}^n$ such that $f(z)=0$. Let $B(0;r)= \left\lbrace z \in \mathbb{C}^n \: : \: \| z \| < r \right\rbrace $ denote the $n$-dimensional open ball centered at the origin $0$ with radius $r>0$. Then the squeezing function $S_\Omega (z)$ of $\Omega$ at $z$ is defined to be
\[ 
S_\Omega (z) 
= 
\sup\limits_{f \in \mathcal{F}_{\Omega}(z)}
\left\lbrace 
\frac{a}{b} 
\: : \:
B(0;a) \subset  f(\Omega ) \subset B(0;b)
\right\rbrace.
\]
 \paragraph{Remark:}
\begin{enumerate}
    \item For the supremum in the definition of the squeezing function, we can restrict the family $\mathcal{F}_{\Omega}(z)$ to the subfamily of functions $f$ such that $f(\Omega)$ is bounded.
    \item For any $\lambda \neq 0$, we have $f \in \mathcal{F}_{\Omega}(z)$ if and only if $\lambda f \in \mathcal{F}_{\Omega}(z)$. As a consequence, we may assume that $b=1$.
\end{enumerate}

It is clear from the definition that the squeezing function on $\Omega$ is positive and bounded above by $1$. Also, it is invariant under biholomorphisms, that is, $S_{g(\Omega)} (g(z)) =S_{\Omega} (z)$ for any biholomorphism $g$ of $\Omega$. 
If the squeezing function of a domain $\Omega$ is bounded below by a positive constant, i.e., if there exists a positive constant $c$ such that $S_\Omega (z) \geq c >0$ for all $z \in \Omega$, then the domain $\Omega$ is said to be \textit{holomorphic homogeneous regular} by Liu, Sun and Yau \cite{liu2004canonical} or with \textit{uniform squeezing property} by Yeung \cite{yeung2009geometry}. The consideration of such domains appears naturally when one applies the Bers embedding theorem to the Teichm\"uller space of genus $g$ hyperbolic Riemann surfaces.    

The squeezing function is interesting because it provides some geometric information about the domain $\Omega$. 
For instance, Joo and Kim proved in  \cite{joo2016boundary} that if $\Omega \subset \mathbb{C}^2$ is a bounded domain with smooth pseudoconvex boundary and if $p \in \partial \Omega$ is of finite type such that $\lim_{\Omega \ni z \to p} S_{\Omega} (z)=1$, then $\partial \Omega$ is strictly pseudoconvex at $p$.
For another instance, Zimmer showed in \cite{zimmer2018gap} and \cite{zimmer2019characterizing} that if $\Omega \subset \mathbb{C}^n$ is a bounded convex domain with $C^{2,\alpha}$ boundary and $K$ is a compact subset of $\Omega$ such that $S_\Omega (z) \geq 1 - \epsilon$ for every $z \in \Omega \setminus K$ and for some positive constant $\epsilon = \epsilon (n)$, then $\Omega$ is strictly pseudoconvex. 
In addition to providing geometric information, the squeezing function is related to some estimates of intrinsic metrics on $\Omega$. 
For example, in \cite{deng2016properties}, Deng, Guan and Zhang showed that  
\[ S_\Omega (z) K_\Omega (z, v) \leq C_\Omega (z,v) \leq  K_\Omega (z,v) \]
for any point $z$ in $\Omega$ and for any tangent vector $v \in T_z{\Omega}$,  
where $C_\Omega$ and $K_\Omega$ denote the Carath\'{e}odory seminorm and Kobayashi seminorm on $\Omega$ respectively. 
For other properties and applications of squeezing functions, see \cite{squ_def1,fornaess2016estimate,fornaess2018domain,fornaess2015estimate,fornaess2016non,kim2016uniform,nikolov2018behavior,nikolov2017boundary,zimmer2018smoothly}.

Given a bounded domain $\Omega \subset \mathbb{C}^n$, it is then natural to ask whether one can estimate or even compute the precise form for the squeezing function $S_\Omega (z)$ on $\Omega$. In \cite{arosio2017squeezing},
Arosio, Forn{\ae}ss, Shcherbina and Wold  provided an estimate of $S_\Omega (z)$ for $\Omega= \mathbb{P}^1 \backslash K$ where $K$ is a Cantor set. In \cite{squ_def1}, Deng, Guan and Zhang showed that the squeezing functions of classical symmetric domains are certain constants (using a result of Kubota in \cite{kubota1982note}); they also  showed that the squeezing function of the $n$-dimensional punctured unit ball $B(0;1) \setminus \{0\}$ is given by $S_{B(0;1) \setminus \{0\}} (z)=\| z \| $. 

We now consider the $n=1$ case, and introduce the following notation which will be used in this paper. $\mathbb{D}_{r}= \{z\in\mathbb{C}:|z|<r\}$, the disk of radius $r$ centered at $0$, and $\mathbb{D}=\mathbb{D}_{1}$; $C_{r}= \{z\in\mathbb{C}:|z|=r\}$, the circle of radius $r$ centered at $0$; and $A_{r}= \{z\in\mathbb{C}:r<|z|<1\}$, the annulus with inner radius $r$ and outer radius $1$.

By the Riemann mapping theorem, the simply-connected case is trivial: $S_{D}(z)\equiv 1$ for any simply-connected domain $D$. In \cite{squ_def1},
 Deng, Guan and Zhang  considered the squeezing function of an annulus $A_r$. They conjectured that for any $z \in A_r$ with $|z| \geq \sqrt{r}>0$,
\[  S_{A_r} (z) = \sigma^{-1} \left( \log \dfrac{(1+|z|)(1-r)}{(1-|z|)(1+r)} \right) \]  
where
\[  \sigma (z)=  \log \dfrac{1+|z|}{1-|z|}.\] 
In this paper, we will disprove this conjecture by establishing the formula for $S_{A_{r}}(z)$. This also answers a question asked by Wold about the precise form for $S_{A_r}(z)$ in his lecture given in the Mini-workshop on Complex Analysis and Geometry at the Institute for Mathematical Sciences, NUS in May 2017.

\begin{theorem}
\label{MainResult1}
For $0<r<1$, and $r<|z|<1$,
\[ S_{A_r} (z) = \max \left\lbrace |z| , \frac{r}{|z|} \right\rbrace. \]
\end{theorem}
\paragraph{Remark:}\begin{enumerate}
\item The case of the punctured disk $A_0=\mathbb{D} \setminus \{0\}$ follows by letting $r \to 0$ so that $S_{A_0} (z) = |z|$. This is the $n=1$ case of the result of Deng, Guan and Zhang for the punctured ball in $\mathbb{C}^n$ referred to above. 
\item Since any doubly-connected domain (other than the punctured plane) is conformally equivalent to $A_{r}$ for some $0\leq r<1$, this result determines the squeezing function in the doubly-connected case up to biholomorphisms.
\end{enumerate}

Let us define
\[ \widetilde{\mathcal{F}}_{r}(z) = \left\lbrace f\in \mathcal{F}_{A_r}(z) : f(A_r) \subset \mathbb{D}, f(\partial \mathbb{D})=\partial \mathbb{D} \right\rbrace, \]
 \[ \widetilde{S}_{r} (z) =  
\sup\limits_{f \in \widetilde{\mathcal{F}}_{r}(z)}
\left\lbrace a 
\: : \:
\mathbb{D}_{a} \subset  f(A_r ) \subset \mathbb{D}
\right\rbrace.
\]
By restricting $\mathcal{F}_{A_{r}}(z)$ to $\widetilde{\mathcal{F}}_{A_{r}}(z)$, we will see in Section \ref{sect:proof} that Theorem \ref{MainResult1} will follow if we show that
\[\widetilde{S}_{r}(z)=|z|.\]
To do this, we will identify a candidate for the extremal function in $\widetilde{\mathcal{F}}_{r}(z)$. This will be the conformal map from $A_r$ onto a circularly slit disk, that is a domain of the form $\mathbb{D} \setminus L$ where $L$ is a proper subarc of the circle with radius $R\in(0,1)$ and center 0. Through the results of Crowdy in \cite{crowdy2005schwarz} and \cite{crowdy2011schottky}, this conformal map can be expressed explicitly in terms of the Schottky-Klein prime function (see Theorem \ref{CrowdyConformal}). It will be shown that, in this case, the radius of the slit is $|z|$. Then the following theorem will show that  this conformal map is indeed extremal (which has been suggested by Wold in his lecture just mentioned).

\begin{theorem}
\label{MainResult3}
Let $\widetilde{E} \subset \mathbb{D}$ be a closed set with $0 \notin \widetilde{E}$ and there exists some constant $y>0$ such that $|z| \geq y$ for any $z \in \widetilde{E}$. Furthermore assume that $\Omega=\mathbb{D} \setminus \widetilde{E}$ is doubly connected. If $g$ is a conformal map of $A_{r}$ onto $\Omega$, for some $r\in(0,1)$, such that $g$ maps the $\partial\mathbb{D}$ onto $\partial\mathbb{D}$, then we have
\[ | g^{-1}(0) | \geq y. \]
\end{theorem}

To prove this result, we will start with the case where $\widetilde{E}$ is a circular arc so that $\Omega=\mathbb{D} \setminus \widetilde{E}$ is a circular slit disk. We will then grow a curve from the circular slit so that $\widetilde{E}$ is now the union of the curve with the circular slit; the conformal maps onto these domains $\Omega$ will then satisfy a version of the Loewner differential equation due to Komatu (see \cite{Komatu_proof,Komatu_origin}). Studying this differential equation will enable us to prove Theorem \ref{MainResult3}. The remaining cases for $E$ then follow by letting the length of the circular slit tend to $0$.

\ \\
Finally, we obtain the following lower bound for the squeezing function on product domains in $\mathbb{C}^{n}$.
\begin{theorem}
\label{thm:several}
Suppose that $\Omega \subset \mathbb{C}^n$ and $\Omega=\Omega_1 \times \cdots \Omega_n$ where $\Omega_i$ is a bounded domain in $\mathbb{C}$ for each $i$. Then for $z=(z_1, \cdots , z_n) \in \Omega$, we have 
\[ 
S_{\Omega} (z) \geq \left( S_{\Omega_1}(z_1)^{-2} + \cdots + S_{\Omega_n}(z_n)^{-2} \right)^{-1/2}.
\]
\end{theorem}

\paragraph{Remark:}
The argument we use to prove the above theorem can be modified to obtain a similar result when $\Omega_{i}$ are not necessarily planar. In \cite{squ_def1}, Deng, Guan and Zhang show that this inequality is attained in the case when each $\Omega_i$ is a  classical symmetric domain.

\hfill

Theorem \ref{thm:several} allows us to use the formula given in Theorem \ref{MainResult1} for the squeezing function of a doubly-connected domain to get a lower bound on the squeezing function of the product of several doubly-connected and simply-connected domains. For example, considering $\Omega = A_r \times \mathbb{D}$, Theorem \ref{thm:several} together with Theorem \ref{MainResult1} yields
\[
S_{A_r \times \mathbb{D}}(z) \geq 
\begin{cases}
\dfrac{r}{\sqrt{r^2+|z_1|^2}}
& \mbox{if } r < |z_1| \leq  \sqrt{r} \\
\dfrac{|z_1|}{\sqrt{1+|z_1|^2}}
& \mbox{if }  \sqrt{r} \leq |z_1| < 1.
\end{cases}
\]
Obtaining the exact form for $S_{A_r \times \mathbb{D}}(z)$ would be of interest.

The rest of the paper is organised as follows.
Firstly, in Section \ref{sect:prelim}, we review some results and concepts that are necessary for this paper including the formula for the conformal map of an annulus $A_r$ to a circularly slit disk $\mathbb{D} \setminus L$ in terms of the Schottky-Klein prime function and a version of the Loewner differential equation that we will need. 
Then, in Section \ref{sec:main_proof1}, we give a proof for Theorem \ref{MainResult3}; the proof of Theorem \ref{MainResult1} is provided in Section \ref{sect:proof} and we prove Theorem \ref{thm:several} in Section \ref{sect:several}. Finally, we discuss the multiply-connected cases in Section \ref{sect:conjecture}.

%==========================
% Preliminary 
%==========================

\section{Preliminary Results}
\label{sect:prelim}

\subsection{Basic Definitions and Notations}
\label{sect:prelim:subsect:notation}
Throughout this paper, we will make use of the following definitions and notations:
\begin{itemize}
    \item For $z \in \mathbb{C}^n$ and $r>0$, $B(z;r)$ denotes the open ball centered at $z$ with radius $r$. When $n=1$, we also set $\mathbb{D}_r=B(0;r)$ and in particular, $\mathbb{D}=B(0;1)$. Then $C_r = \partial \mathbb{D}_r$.
    \item Let $p>0$ and $r=e^{-p}$ so that $0<r<1$. Then $A_r$ denotes the annulus centered at $0$ with inner radius $r$ and outer radius $1$. In this case, $A_r$ is said be of \textit{modulus} $p$.
    \item Let $\Omega$ be a doubly-connected domain in $\mathbb{C}$. The \textit{modulus} $p$ of $\Omega$ is defined to be the unique positive real number $p$ such that there exists a biholomorphism $\phi$ from $\Omega$ to $A_r$ where $r=e^{-p}$. 
    \item By a \textit{(doubly-connected) circularly slit disk}, we refer to a domain of the form $\mathbb{D} \setminus L$ where $L$ is a proper closed subarc of the circle with radius $R \in (0,1)$.
    \item For any set $E\subset\mathbb{C}$, $\partial E$ denotes the topological boundary of $E$ in $\mathbb{C}$. 
    \item Let $\gamma : I \to \mathbb{C}$ be a curve where $I$ is an interval in $\mathbb{R}$. We will write $\gamma I$ instead of $\gamma (I)$ for notational simplicity. 
    \item We assume that the argument function $\mathrm{Arg}$  takes values in $[0,2\pi)$. 
\end{itemize}

\subsection{The Schottky-Klein Prime Function}
\label{sect:prelim:subsect:SKPF}
The Schottky-Klein prime function $\omega(z,y) $ on the annulus $A_{r}$ is defined by 

\begin{equation}
\label{def:Schottky prime}
\omega(z,y) = (z-y) \prod_{n=1}^{\infty} \dfrac{(z-r^{2n}y)(y-r^{2n} z)}{(z-r^{2n} z)(y-r^{2n} y)} 
\qquad
\mbox{for $z,y \in \mathbb{C} \setminus \{ 0\} $. }
\end{equation}
Moreover, $\omega(z,y)$ satisfies the following symmetry properties (see \cite{baker1897abel,crowdy2011schottky} or \cite{hejhal1972theta}):
\begin{equation}
\label{prop1:Schottky prime}
\overline{ \omega ( \overline{z}^{-1} , \overline{y}^{-1} )}
=   \dfrac{- \omega (z, y)}{z y}
\end{equation}
and 
\begin{equation}
\label{prop2:Schottky prime}
\omega ( r^{-2} z , y )
=  \dfrac{ rz \: \omega (z , y ) }{y} .
\end{equation}
Recall from the previous section that a circularly slit disk is a domain of the form $\mathbb{D} \setminus L$ where $L$ is a proper subarc of the circle with radius $R\in(0,1)$ and center 0.
In \cite{crowdy2011schottky}, Crowdy established the following result.

\begin{theorem}
\label{CrowdyConformal}
Let $y$ be a point in $A_r$ and define 
\begin{equation}
\label{ConformalMap}
f(z,y) = \dfrac{\omega(z,y)}{|y| \omega(z,\overline{y}^{-1})}
\qquad
\mbox{for $z,y \in A_r $.}
\end{equation} 
Then $f( \cdot,y)$ is a conformal map from $A_r$ onto a circularly slit disk with $f(\partial\mathbb{D})=\partial\mathbb{D}$ and $y$ is mapped to $0$.
\end{theorem}

See also \cite{book_crowdy2020}.
Theorem \ref{CrowdyConformal} allows us to compute the radius of the circular arc in $f(A_r)$.
For any $z=re^{i \theta}$, we have
\begin{align*}
| f(z,y) |^2
&=
f(z,y) \overline { f(z,y)} \\
&=
\left( 
\dfrac{\omega(z,y)}{|y| \omega(z,\overline{y}^{-1})}
\right) 
\left( 
\dfrac{ \overline {\omega (z,y)} }{|y| \overline {\omega(z,\overline{y}^{-1})} }
\right) \\
&=
\dfrac{1}{|y|^2}
\left( 
\dfrac{\omega(z,y)}{ \omega(z,\overline{y}^{-1})}
\right) 
\left( 
\dfrac{ \overline {\omega ( r^2 \overline{z}^{-1}, y )} }{ \overline {\omega ( r^2 \overline{z}^{-1}, \overline{y}^{-1}) }}
\right).
\end{align*}
Using (\ref{prop1:Schottky prime}) and (\ref{prop2:Schottky prime}),
\begin{align*}
| f(z,y) |^2&=
\dfrac{1}{|y|^2}
\left( 
\dfrac{\omega(z,y)}{ \omega(z,\overline{y}^{-1})}
\right) 
\left( 
\dfrac{r^{-2}z y  \: \omega ( r^{-2} z, \overline{y}^{-1} ) }{ r^{-2}z \overline{y}^{-1} \: \omega ( r^{-2} z, y) }
\right) \\
&=
\left( 
\dfrac{\omega(z,y)}{ \omega(z,\overline{y}^{-1})}
\right) 
\left( 
\dfrac{ rz y \: \omega ( z, \overline{y}^{-1} ) }{ rz \overline{y}^{-1} \: \omega ( r^{-2} z, y) }
\right) \\&=
|y|^2.
\end{align*}
This shows that the radius of the circular arc in $f(A_r)$ is $|y|$ and, in particular, it does not depend on $r$. Note that if $\phi$ is a conformal map from a circularly slit disk $\Omega_1$ to another circularly slit disk $\Omega_2$ such that $\phi$ maps $\partial\mathbb{D}$ to $\partial\mathbb{D}$ and $\phi(0)=0$, then $\phi$ must be a rotation (Lemma 6.3 in \cite{conway2012functions}).  We restate this result as the following lemma.

\begin{lemma}
\label{Lemma1}
For any annulus $A_r$ and for any circularly slit disk $\Omega$ whose arc has radius $y$, if $f$ is a conformal map which maps $A_r$ onto $\Omega$ with $f(\partial \mathbb{D}) = \partial \mathbb{D}$, then we have 
\[ |f^{-1} (0)| =y. \]
\end{lemma}
\paragraph{Remark:} We thank the referee for informing us that Lemma \ref{Lemma1} is the same as Lemma $3$ of \cite{reich1960canonical}. Our proof of the above lemma is different from that of Lemma $3$ in \cite{reich1960canonical}.

When $y$ is positive, we have the following lemma. 
\begin{lemma}
\label{lemma:sym}
Let $\Omega=\mathbb{D} \setminus L$ be a circularly slit disk such that $L$ is symmetric across the real axis, that is, $\overline{z} \in L$ if and only if $z \in L$. Let $y$ be a point on an annulus $A_r$ such that $r<y<1$. Let $g: A_r \to \Omega$ be the conformal map from $A_r$ onto $\Omega$ such that $g(y)=0$ and $g(\partial \mathbb{D}) =\partial \mathbb{D}$. Then for any $z \in A_r$, 
\[g(z)=\overline{g(\overline{z})}.\]
In particular, if $p_1$ and $p_2$ denote the two end points of $L$ and $\pi_1,\pi_2 \in C_r$ denote the two points on the inner boundary of $A_r$ such that $g(\pi_1)=p_1$ and $g(\pi_2)=p_2$; then we have $\overline{\pi_1} = \pi_2$. 
\end{lemma}
\begin{proof}
 Consider the map $G$ defined by $G(z)=\overline{g(\overline{z})}$. Then $G$ is a conformal map of $A_{r}$ onto $\Omega$ with $G(\partial\mathbb{D})=\partial\mathbb{D}$ and $G(y)=0$. Hence, $G^{-1}\circ g$ is a conformal automorphism of $A_{r}$ that fixes $y$ and does not interchange the boundary components. Since any such conformal automorphism of an annulus is the identity mapping, this implies that $G^{-1} \circ g$ is the identity and so $g(z)=\overline{g(\overline{z})}$ on $A_r$. Let $\{ p_{i,n}\} \subset \Omega$ be sequences of points such that $p_{2,n}=\overline{p_{1,n}} $ and  $\lim\limits_{n \to \infty} p_{i,n} = p_i$. Define $\pi_{i,n}=g^{-1} (p_{i,n}) \in A_r$ for $i=1,2$. It follows that for each $n\in \mathbb{N}$, 
\[ g(\pi_{2,n}) = p_{2,n} = \overline{p_{1,n}} = \overline{ g(\pi_{1,n}) }= g (\overline{\pi_{1,n}} ) .\]
Since $g$ is conformal, we have $\overline{\pi_{1,n}}=\pi_{2,n}$ for each $n\in \mathbb{N}$. Note that $\lim\limits_{n \to \infty} \pi_{i,n} = \pi_i$. Hence we conclude that $\overline{\pi_1} =\pi_2$.
\end{proof}

%========================================
% Preliminary - Kernel
%========================================
\subsection{Approximating by Slit Domains}
\label{sect:prelim:subsect:Caratheodory}

Let $\mathcal{D}$ be the collection of doubly-connected domains $\Omega$ so that $\Omega \subset A_r$ for some $r>0$ and $\partial \mathbb{D}$ be one of the boundary components of $\Omega$. Let $\{ \Omega_n \}$ be a sequence in $\mathcal{D}$. Define the \textit{kernel} $\Omega$ of the sequence $\{ \Omega_n \}$ as follows:
\begin{itemize}
\item If there exists some $\Omega^* \in \mathcal{D}$ such that $\Omega^* \subset \bigcap\limits_{n=1}^{\infty} \Omega_n$, then the kernel $\Omega$ is defined to be the maximal doubly-connected domain in $\mathcal{D}$ such that for any compact subset $K$ of $\Omega$, there exists an $N \in \mathbb{N}$ so that $K \subset \Omega_n$ whenever $n>N$;
\item otherwise, the kernel $\Omega$ is defined to be $\partial \mathbb{D}$.
\end{itemize} 
Then a sequence $\{ \Omega_n \}$ converges to $\Omega$ \textit{in the sense of kernel convergence}, if $\Omega$ is the kernel of every subsequence of $\{ \Omega_n \}$. 
Let $\{ \Omega_n \}$ be a sequence of doubly-connected domains in $\mathcal{D}$. Since every doubly-connected domain of finite modulus is conformally equivalent to an annulus $A_{r}$ for some $r\in(0,1)$, there exists a sequence $\{ \psi_n \}$ of conformal maps such that $\psi_n$ maps $A_{r_n}$ onto $\Omega_n$ and $\psi_n$ is normalized appropriately. A version of the Carath\'{e}odory kernel convergence theorem for doubly-connected domains  will show that the kernel convergence of $\{ \Omega_n \}$ implies local uniform convergence of $\{ \psi_n \}$. This is Theorem 7.1 in \cite{Komatu_proof}. We restate this result in a form which we will need later.

\begin{theorem}
\label{dense2}
Suppose that $r>0$ and $r<y<1$. Let $\{ \Omega_n \}$ be a sequence of doubly connected domains in $\mathcal{D}$ such that $y\in \bigcap_{n=1}^\infty \Omega_n$.
Let $\{ r_n \}$ be a sequence with $r< r_n <1$ for $n \geq 1$ such that there exists a conformal map $\Phi_n$ of $\Omega_n$ onto $A_{r_n}$ satisfying $\Phi_n (y) > 0$ and $\Phi_n (\partial \mathbb{D})= \partial \mathbb{D}$ for every $n$. Then the kernel convergence of $\Omega_n$ to a doubly connected domain $\Omega$ in $\mathcal{D}$ implies that the sequence $\{r_n\}$ converges to $r$ and that the sequence $\{ \Phi_n \}$ converges locally uniformly to a conformal map $\Phi$ of $\Omega$ onto $A_r$ satisfying $\Phi(y)>0$ and $\Phi (\partial \mathbb{D})= \partial \mathbb{D}$. 
\end{theorem}
Theorem \ref{dense2} can be obtained from Theorem 7.1 in \cite{Komatu_proof} by renormalising the conformal maps. This theorem leads to the following proposition.
\begin{proposition}
\label{dense}
Let $E \subset \overline{A}_r$ be a closed set such that $A_r \setminus E$ is doubly connected and $E \cap C_r \neq \emptyset$. 
Assume there exists some $y \in A_r \setminus E$ with $y>0$ and let $\Phi$ be the conformal map from $A_r \setminus E$ onto some annulus  $A_{r'}$ normalized such that $\Phi (y) >0$ and $\Phi (\partial \mathbb{D})= \partial \mathbb{D}$. 
\begin{enumerate}
\item Suppose that $\partial E \cap A_{r}$ is a Jordan arc. Then we can find a Jordan arc $\gamma : [0,T) \to \overline{A_r}\setminus\{y\}$ satisfying $\gamma (0) \in C_r$ and $\gamma (0,T) \subset A_r$ and an increasing function $q:[0,T) \to [r,r']$ with $q(0)=r$ and $q(T)=r'$ such that the conformal maps $\Phi_t$ of $A_r \setminus \gamma (0,t] $ onto $A_{q(t)}$ with $\Phi_{t} (y)>0$ and $\Phi_t (\partial \mathbb{D})= \partial \mathbb{D}$ satisfy $\Phi_t \to  \Phi$ locally uniformly as $t \to T$. 
\item For the cases where $\partial E \cap A_r$ is not a Jordan arc, we can find an increasing sequence  $\lbrace q_{n} \rbrace$ with $r<q_{n}<1$ for all $n$ and $q_{n}\to r'$ as $n\to \infty$; a sequence of Jordan arcs $\Gamma_{n}\subset \overline{A_{r}} \setminus \{y \}$ which starts from $C_{r}$; conformal maps $\Phi_{n}$ which map $A_r\setminus\Gamma_{n}$ onto  the annulus $A_{q_{n}}$ with $\Phi_{n}(y)>0$  and $\Phi_n (\partial \mathbb{D})= \partial \mathbb{D}$ such that $\Phi_n \to  \Phi$ locally uniformly as $n \to \infty$. 
\end{enumerate}
\end{proposition}

\begin{proof}
For part 1, suppose that $\partial E \cap C_r \neq \emptyset$, then we can find a Jordan arc $\gamma:[0,T] \to \overline{A_r}$  such that $|\gamma(0)|=r$ and $\gamma(0,T) =\partial E \cap A_r$
Otherwise, $\partial E \cap C_r = \emptyset$. Then $E$ is bounded by $C_r$ and a closed curve in $A_r$. In this case, we define a Jordan arc $\gamma:[0,T] \to \overline{A_r}$  such that $\gamma [0,T_1]$ is the straight line segment in $E$ from a point in $C_r$ to a point in $\partial E$ for some $0<T_1<T$ and $\gamma [T_1,T) =\partial E \cap A_r$. We define $\Omega_t = A_r \setminus \gamma (0,t] $. 
The conformal equivalence of any doubly-connected domains to an annulus implies that there exists an increasing function $q:[0,T] \to [r,r']$ with $q(0)=r$ and $q(T)=r'$ and a family of conformal maps $\Phi_t$ of $\Omega_t$ onto $A_{q(t)}$ with $\Phi_t (y) > 0$ and $\Phi_t (\partial \mathbb{D})= \partial \mathbb{D}$. Then as $t \to T$, $ \Omega_t \to A_r \setminus E$ in the sense of kernel convergence. Hence, by Theorem \ref{dense2}, the sequence $\{ \Phi_t\}$ converges locally uniformly to a conformal map $\Phi$ of $A_r \setminus E$ onto $A_{r'}$ such that $\Phi (y) > 0$ and $\Phi (\partial \mathbb{D})= \partial \mathbb{D}$.  This proves part 1. 

For part 2, since $A_r \setminus E$ is doubly connected, there exists an annulus $A_s$ and a conformal map $f$ of $A_s$ onto $A_r \setminus E$ for some $s>0$ such that $f(\partial \mathbb{D}) = \partial \mathbb{D}$. For any $0< \delta < 1-s$, we let \[E_\delta = \overline{f(\{ z \in A_s :  s< |z| < s+ \delta \}) }.\] Also, when $\delta$ is small enough, we have $y \notin E_\delta$. Since $f$ is conformal, $\partial E_{\delta} \cap A_r = f(C_{s+\delta})$ is an analytic Jordan arc. So part 1 of the proposition applies to $A_r \setminus E_{\delta}$. 
That is, we can find a Jordan arc $\gamma_{\delta} : [0,T) \to \overline{A_r}\setminus\{y\}$ satisfying $\gamma_\delta (0) \in C_r$ and $\gamma_\delta (0,T) \subset A_r$ and an increasing function $q_\delta :[0,T) \to [r,r']$ with $q_{\delta}(0)=r$ and $q_{\delta}(T)=r'_\delta$ such that the conformal maps $\Phi_{t}^{\delta}$ of $A_r \setminus \gamma_{\delta} (0,t] $ onto $A_{q_{\delta}(t)}$ with $\Phi_{t}^{\delta} (y)>0$ and $\Phi_{t}^{\delta} (\partial \mathbb{D})= \partial \mathbb{D}$ satisfy $\Phi_{t}^{\delta} \to  \Phi^\delta$ locally uniformly as $t \to T$. 
Letting $\delta \to 0$, and applying a diagonal argument we get the desired result. 
More precisely, let $\{ t_n \} \subset [0,T)$ be an increasing sequence and the desired result will follow by letting $\Gamma_n =  \gamma_{1/n} (t_n)$ and $\Phi_n = \Phi_{t_n}^{1/n}$. 

\end{proof}

The simply-connected version of this proposition is given in Theorem 3.2 in \cite{duren2001univalent}. This proposition allows us to consider slit domains in the proof of Theorem \ref{MainResult3} via a Loewner-type differential equation which is introduced in the next section. This is analogous to the approach to solving various coefficient problems for univalent functions on $\mathbb{D}$ (including the De Branges-Bieberbach Theorem); see  \cite{duren2001univalent} and references therein.

%=====================================================
% Preliminary - LDE
%=====================================================
\subsection{The Loewner-type Differential Equation in Annuli}
\label{sect:prelim:subsect:LDE}
In this section, we introduce the Loewner differential equation which is a differential equation for the conformal maps (normalized and parametrized appropriately) onto a slit domain. We first introduce the classical setting in annuli, where the slit grows from the outer boundary.

For $p>0$ , define $r=e^{-p}$ and $r_t=e^{-p+t}$. Suppose that $\gamma : [0,T] \to \overline{A_r}$ is a Jordan arc satisfying $\gamma (0) \in \partial \mathbb{D}$ and $\gamma (0,T] \subset A_r$, such that $A_r \setminus \gamma (0,t]$ has modulus $p-t$. Then there exists a family of conformal maps $\phi_t : A_r \setminus \gamma (0,t] \to A_{r_t}$, continuously differentiable in $t$, such that 
\[  \alpha (t)  := \phi_t ( \gamma (t) ) \in \partial \mathbb{D} \]
and  $\phi_t (z)$ satisfies the Komatu's version of the Loewner Differential Equation on an annulus (see \cite{Komatu_proof}),
\begin{equation}
\label{LDE_out}
\partial_t \phi_t (z) =  \phi_t (z)  \mathcal{K}_{r_t} (\phi_t (z),\alpha (t) ) .
\end{equation}
Here, for $\alpha \in \partial \mathbb{D}$ and $r>0$, $\mathcal{K}_r(z,\alpha)$ is the Villat's kernel, defined by $\mathcal{K}_r(z,\alpha) = \mathcal{K}_r \left( \frac{z}{\alpha} \right),$ where
\[ \mathcal{K}_r(z) = \lim\limits_{N \to \infty} \sum\limits^N_{n=-N} \dfrac{r^{2n}+z}{r^{2n}-z} . \]

For our purposes, we need a version of Loewner-type differential equations where the curve grows from the inner boundary circle of $A_{r_0}$.  Let $C_{r_0}$ be the circle centered at $0$ with radius $r_0$. Suppose that $\gamma : [0,T] \to \overline{A_{r_0}}$ is a Jordan arc satisfying $\gamma (0) \in C_{r_0}$ and $\gamma (0,T] \subset A_{r_0}$ such that $A_{r_0} \setminus \gamma (0,t]$ has modulus $p-t$.  Define the inversion map $\rho_t (z)=\dfrac{r_t}{z}$ which is a conformal automorphism of $A_{r_{t}}$ that interchanges inner boundary and outer boundary of $A_{r_t}$. Hence $\rho_t \circ \gamma $ is a Jordan arc satisfying the conditions given at the beginning of this subsection and for this Jordan arc $\rho_t \circ \gamma $, let $\phi_t$ be the corresponding conformal map satisfying (\ref{LDE_out}). Then we define 
\[ \Phi_t(z):= \rho_t \circ \phi_t \circ \rho^{-1}_0 (z) =  \dfrac{r_t}{\phi_t  \left( \dfrac{r_0}{z} \right)} . \]
Clearly $\Phi_t $ is a conformal map from $A_{r_0} \setminus \gamma (0,t]$  onto $A_{r_t}$ with $\Phi_t (\gamma (t)) \in C_{r_t}$. By the chain rule, $\Phi_t$ satisfies the differential equation 
\begin{equation*}
 \partial_t \Phi_t (z)
=  \Phi_t (z)  \left( 1-  \mathcal{K}_{r_t} \left( \dfrac{r_t}{\Phi_t (z)},  e^{-i \beta (t) }  \right) \right)
\end{equation*}
where $\beta (t) $ satisfies $\Phi_t ( \gamma (t) ) = r_t e^{i \beta (t)} \in \overline{A_{r_t}} $.

Let $y_0>0$ be a fixed point in $A_{r_0}$. By composing $\Phi_t$ with a suitable rotation, we can normalize $\Phi_t$ such that $y_t:=\Phi_t(y_0)>0$ for any $t$. Then $\Phi_t $ satisfies
\begin{align}
 \partial_t \Phi_t (z)
&=  \Phi_t (z)  \left( 1-  \mathcal{K}_{r_t} \left( \dfrac{r_t}{\Phi_t (z)},  e^{-i \beta (t) }  \right)  + i  J(r_t,y_t, \beta (t) ) \right) 
\label{eq:LDE_in}
\\
&=  \Phi_t (z)  \left( 1-  \lim\limits_{N \to \infty} \sum\limits_{n=-N}^N \dfrac{r_t^{2n-1} \Phi_t (z) + e^{i\beta (t) } }{r_t^{2n-1} \Phi_t (z) - e^{i\beta (t) } }  + i  J(r_t,y_t, \beta (t) ) \right) 
\nonumber
\end{align}
for some function $ J(r_t,y_t, \beta (t) )$. Notice that the normalization accounts for  multiplying a rotation factor to $\Phi_t$ at each time $t$. So the function $ J(r_t,y_t, \beta (t) )$ is real-valued. Also, since $y_t >0$ for any $t$, we have 
\[ \partial_t \mathrm{Im} \left( \log y_t \right) =0 \]
and hence we have 
\[ J(r_t,y_t, \beta (t) ) = \mathrm{Im} 
 \left[ \mathcal{K}_{r_t} \left( \dfrac{r_t}{y_t}, e^{-i\beta (t) } 
\right) \right]. \]
We call the function $\beta(t)$ the \textit{Loewner driving function }of the curve $\gamma$. We now define 
\[
Q(r,y,\theta,w) :=  1-  \lim\limits_{N \to \infty} \sum\limits_{n=-N}^N \dfrac{r^{2n-1} w + e^{i \theta} }{r^{2n-1} w - e^{i\theta} }  + i  J(r,y, \theta),
\]

\[
R (r,\theta; w) := 
\mathrm{Re}\left[ 
Q(r,y,\theta,w)
\right],
\qquad 
\mbox{$r>0$}
\]
and
\[
I (r,y,\theta;w) := 
\mathrm{Im}\left[ 
Q(r,y,\theta,w)
\right],
\qquad 
\mbox{$r>0$, $y>0$}.
\]
Hence
\[ \frac{ \partial_t \Phi_t (z)}{\Phi_t (z)}  = Q(r_t,y_t,\beta (t) , \Phi_t (z))  = R(r_t,\beta (t) , \Phi_t (z))   +i I (r_t,y_t,\beta (t), \Phi_t (z))   \]
Substituting $z=y_0$ in the above equation and noting that $y_t =\Phi_t (y_0)>0$, we get 
\begin{equation}
\label{LDE}
\partial_t \log y_t = P(r_t , y_t, \beta (t) )
\end{equation}
where
\[
P (r,y,\theta) := R(r,\theta,y)
=
\mathrm{Re}\left[ 
1 - \lim\limits_{N \to \infty} \sum\limits_{n=-N}^N \dfrac{r^{2n-1} y + e^{i \theta }} { r^{2n-1} y - e^{i \theta} }
\right]
\]
for $0<r < y<1$.

\subsection{Multi-slit Loewner-Type Differential Equation}
\label{subsec:Prelim-Muti}
%========================================
% Preliminary - Multi
%========================================

In this subsection, we will develop a version of the Loewner differential equation for multiple slits on an annulus.  

Write $r_0=e^{-p}$. Let $y_0$ be a point in $A_{r_0}$, $\gamma_1:[0,T_1] \to \overline{A_{r_0}}$ and $\gamma_2:[0,T_2] \to \overline{A_{r_0}}$ be Jordan arcs such that $\gamma_1[0,T_1] \cap \gamma_2[0,T_2] =\emptyset$. 
Moreover, $\gamma_1$ and $\gamma_2$ are parametrized such that  $A_{r_0} \setminus \left( \gamma_1(0,t_1] \cup \gamma_2(0,t_2] \right)$ has modulus $p-|\tau|$, where $\tau = ( t_1 , t_2)$ and $|\tau| := t_1 +t_2 $. 

We now make the following construction which is illustrated in Figure \ref{fig:fig1}. 
\begin{itemize} 

\item Let $y_{(0,0)}:=y_0$.

\item Let $\widetilde{\Phi}_{(t_1,0)}$ be the conformal map of $A_{r_0} \setminus \gamma_1(0,t_1]$ onto $A_{r_{t_1}}$ with  $\widetilde{y}_{(t_1,0)} := \widetilde{\Phi}_{(t_1,0)} (y_{(0,0)}) >0 $.

\item Let $\widehat{\Phi}_{(0,t_2)}$ be the conformal map of $A_{r_0} \setminus \gamma_2(0,t_2]$ onto $A_{r_{t_2}}$ with  $\widehat{y}_{(0,t_2)} := \widehat{\Phi}_{(0,t_2)} (y_{(0,0)}) >0 $.

\item Let $\widetilde{\Phi}_{\tau}$ be the conformal map of $A_{r_{t_2}} \setminus \widehat{\gamma}_1(0,t_1]$ onto $A_{r_{|\tau|}}$ with  $\widetilde{y}_{\tau} = \widetilde{\Phi}_{\tau} (\widehat{y}_{(0,t_2)}) >0 $.
Here $\widehat{\gamma}_1= \widehat{\Phi}_{(0,t_2)} \circ \gamma_1$ is the image of $\gamma_1(0,t_1]$ under $\widehat{\Phi}_{(0,t_2)}$. Note that because of the conformal invariance and the fact that $A_{r_0} \setminus \left( \gamma_1(0,t_1] \cup \gamma_2(0,t_2] \right)$ has modulus $p-|\tau|$, we have $A_{r_{t_2}} \setminus \widehat{\gamma}_1 (0,t_1]$  has modulus $p-|\tau|$.

\item Let $\widehat{\Phi}_{\tau}$ be the conformal map of $A_{r_{t_1}} \setminus \widetilde{\gamma}_2(0,t_2]$ onto $A_{r_{|\tau|}}$ with  $\widehat{y}_{\tau} = \widehat{\Phi}_{\tau} (\widetilde{y}_{(t_1,0)}) >0 $. 
Here $\widetilde{\gamma}_2= \widetilde{\Phi}_{(t_1,0)} \circ \gamma_2$ is the image of $\gamma_2 (0,t_2]$ under $ \widetilde{\Phi}_{(t_1,0)}$. Note that because of the conformal invariance and the fact that $A_{r_0} \setminus \left( \gamma_1(0,t_1] \cup \gamma_2(0,t_2] \right)$ has modulus $p-|\tau|$, we have $A_{r_{t_1}} \setminus \widetilde{\gamma}_2 (0,t_2]$  has modulus $p-|\tau|$.

\item Let $\Phi_{\tau} $ be the conformal map of $A_{r_0} \setminus  \left( \gamma_1 (0,t_1] \cup \gamma_2 (0,t_2] \right)$ onto $A_{r_{|\tau|}}$ with $y_{\tau} := \Phi_{\tau} (y_{(0,0)})$.

\item Let $r_{|\tau|} e^{i \xi_1 (\tau) } = \Phi_{\tau} ( \gamma_1 (t_1))$ and $r_{|\tau|} e^{i \xi_2 (\tau)} = \Phi_{\tau} ( \gamma_2 (t_2))$.
\end{itemize}

\begin{figure}
\centering
\includegraphics[width=12cm]{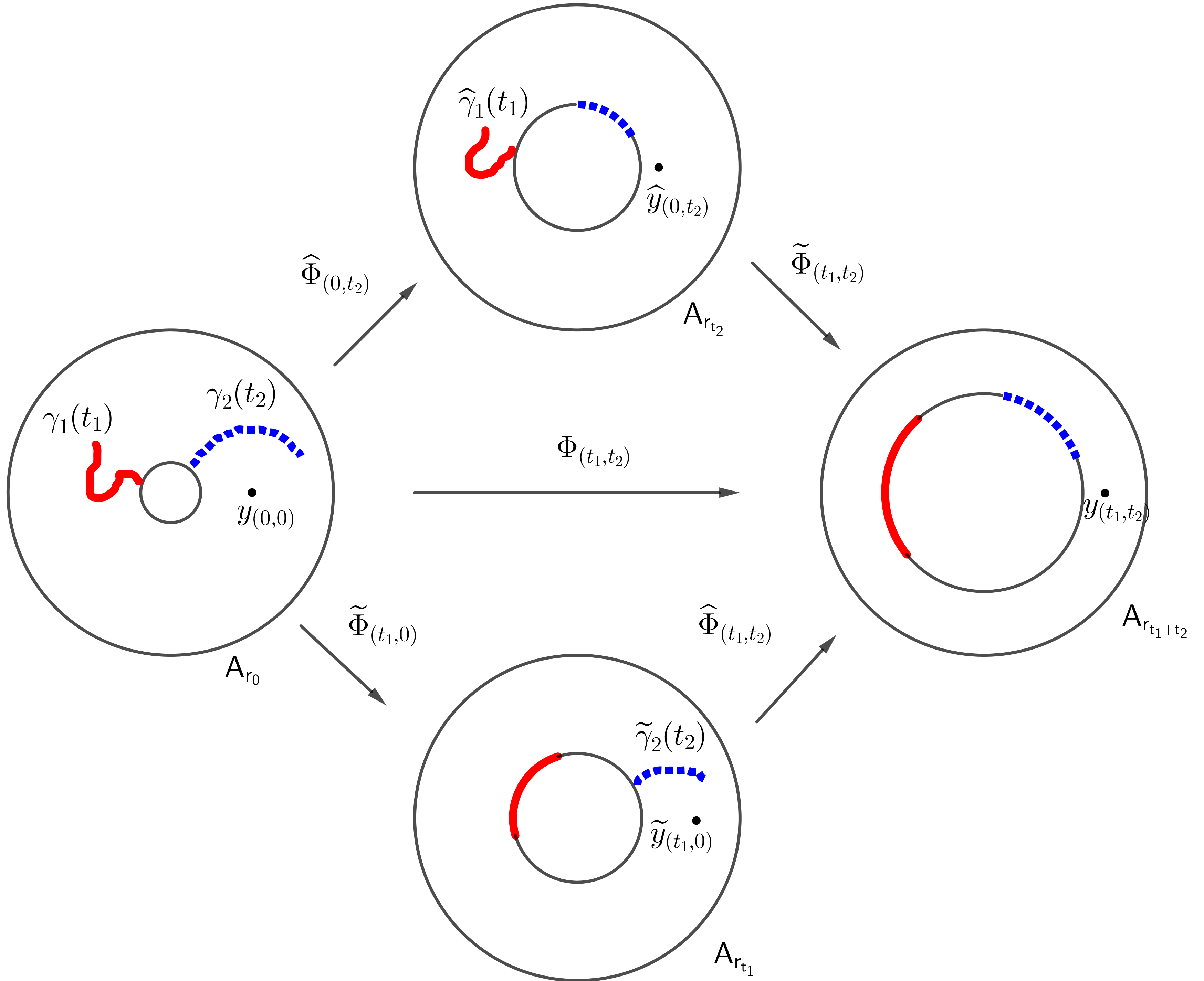}
\caption{Construction of two slit Loewner differential equation.}
\label{fig:fig1}
\end{figure}

Note that the only conformal automorphism of an annulus which fixes a point and does not interchange the boundary components is the identity mapping. Hence, we have
\begin{equation}
\label{Eq3}
\Phi_{\tau} = \widehat{\Phi}_{\tau} \circ \widetilde{\Phi}_{(t_1,0)} = \widetilde{\Phi}_{\tau} \circ \widehat{\Phi}_{(0,t_2)} 
\end{equation}
and 
\[ y_{\tau} = \Phi_{\tau} (y_{(0,0)}) = \widehat{\Phi}_{\tau} (\widetilde{y}_{(t_1,0)}) =  \widetilde{\Phi}_{\tau} (\widehat{y}_{(0,t_2)}). \]
Then $\widetilde{\Phi}_{\tau}$ satisfies (\ref{eq:LDE_in})
\[ \partial_{t_{1}} \widetilde{\Phi}_{\tau}(w) = \widetilde{\Phi}_{\tau} (w) 
Q( r_{|\tau|} , y_\tau  ,\xi_1 (\tau) , \widetilde{\Phi}_{\tau}(w)). \]
By substituting $w=\widehat{\Phi}_{(0,t_{2})}(z)$ and $w=\widehat{y}_{(0,t_2)}$ respectively, and by (\ref{Eq3}) we have
\[\partial_{t_{1}} \Phi_{\tau}(z) = \Phi_{\tau} (z) 
Q( r_{|\tau|} , y_\tau  ,\xi_1 (\tau) , \Phi_{\tau}(z))\]
and
\begin{equation}
\label{LDE:tau-t2}
\partial_{t_1} \log  y_{\tau}  =
 P(r_{|\tau|},y_{\tau}, \xi_1 (\tau)).
\end{equation}
Similarly, $\widehat{\Phi}_{\tau}$ also satisfies (\ref{eq:LDE_in}), 
\[ \partial_{t_2}\widehat{\Phi}_{\tau}(w) = \widehat{\Phi}_{\tau}(w) 
Q( r_{|\tau|} , y_\tau  ,\xi_2 (\tau) , \widehat{\Phi}_{\tau}(w)). \]
Substituting $w=\widetilde{\Phi}_{(t_{1},0)}(z)$ and $w=\widetilde{y}_{(t_1,0)}$ respectively, we have 
\[\partial_{t_2}\Phi_{\tau}(z) = \Phi_{\tau}(w) 
Q( r_{|\tau|} , y_\tau  ,\xi_2 (\tau) , \Phi_{\tau}(w))\]
and
\begin{equation}
\label{LDE:tau-t1}
\partial_{t_2}\log  y_{\tau} 
= P(r_{|\tau|},y_{\tau},\xi_2 (\tau)).
\end{equation}

Now let $\gamma_3:[0,T_3] \to \overline{A_{r_0}}$ be another Jordan arc such that $\gamma_3(0,T_3] \cap \gamma_2(0,T_2] = \emptyset$ and $\gamma_3(0,T_3] \cap \gamma_1(0,T_1] = \emptyset$. Moreover, $\gamma_1,\gamma_2$ and $\gamma_3$ are parametrized such that $A_{r_0} \setminus \left(  \gamma_1(0,t_1]  \cup \gamma_2(0,t_2] \cup \gamma_3(0,t_3] \right)$ has modulus $p-|\tau|$, where $\tau =(t_1,t_2,t_3)$ and $|\tau| = t_1+t_2+t_3$. A similar construction to the above allows us to find a family of conformal maps 
\[ \Phi_{\tau}: A_{r_0} \setminus \left(  \gamma_1(0,t_1]  \cup \gamma_2(0,t_2] \cup \gamma_3(0,t_3]  \right) \to A_{r_{|\tau|}}\] 
with $y_\tau :=\Phi_\tau (y_{(0,0,0)}) >0$, where $y_{(0,0,0)}=y_0$. These satisfy
\[\partial_{t_i}\Phi_{\tau}(z) = \Phi_{\tau}(w) 
Q( r_{|\tau|} , y_\tau  ,\xi_i (\tau) , \Phi_{\tau}(w))\]
and
\[\partial_{t_i}\log  y_{\tau} 
= P(r_{|\tau|},y_{\tau},\xi_i (\tau))\]
for $i=1,2,3$. Moreover, suppose that $t_1,t_2,t_3$ are real-valued functions of $s$, that is $t_1=t_1(s)$, $t_2=t_2(s)$ and $t_3=t_3(s)$. By the chain rule, $\Phi_\tau$ and $y_\tau$ satisfies
\begin{equation}
\label{MLDE}
\partial_{s} \log \Phi_{\tau} (z)
=   \sum\limits_{i=1}^3  (\partial_s t_i) Q \left( r_{|\tau|},y_\tau, \xi_i (\tau) , \Phi_{\tau} (z) \right) 
\end{equation}
and 
\begin{equation}
\label{MLDE2}
\partial_{s}\log  y_{\tau} 
= \sum\limits_{i=1}^3  (\partial_s t_i) P(r_{|\tau|},y_{\tau},\xi_i (\tau) ) 
\end{equation}

%==========================
% 
%==========================

\section{Proof of Theorem \ref{MainResult3}}
\label{sec:main_proof1}

\subsection{Idea of the Proof}
\label{subsec:idea_of the proof}
Suppose that $y>0$. Let $E\subset\mathbb{D}$ be a closed set with $0\not\in E$ and $|z|\geq y$ for all $z\in E$. Let $g$ be a conformal map of an annulus $A_{r}$ onto $\mathbb{D}\setminus E$. By further composing with a rotation we can assume that $g^{-1}(0) > 0$. We need to show that $g^{-1}(0)\geq y$. 

We first consider the case where $E$ is the union of a circular arc $L$ (with radius $y$ centred at $0$) and a Jordan arc starting from $L$. Proposition \ref{dense} will allow us to obtain the general case by an approximation argument.

Denote by $f$ the conformal map of the annulus $A_{{r}_{0}}$ onto a circularly slit domain $\mathbb{D}\setminus L$ which maps a point $y\in A_{r_{0}}$ with $y>0$ to $0$. Let $\widetilde{\gamma}$ be a Jordan arc growing from the circular arc $L$. In other words,
$\widetilde{\gamma} : [0,T] \to \mathbb{D}$ is a Jordan arc satisfying $\widetilde{\gamma} (0) \in L$ and $\widetilde{\gamma} (0,T] \subset \Omega$, such that $\Omega \setminus \widetilde{\gamma} (0,t]$ has modulus $-\log r_t = -(\log r_0)-t$. Now, we let $\gamma = f^{-1} \circ \widetilde{\gamma}$. Let $y_0=y$ and we then define $y_{t}$ and $\beta(t)$ as in Section \ref{sect:prelim:subsect:LDE}. 

Now let $f_{t}$ be the conformal map of the annulus $A_{{r}_{t}}$ onto a circularly slit domain $\mathbb{D} \setminus L_t$, where $L_t$ is a circular arc, which maps the point $y_{t}\in A_{r_{t}}$ to $0$. The maps $f_{t}$ can each be extended continuously to the inner circle $C_{r_{t}}$ of $A_{r_{t}}$ and $f_{t}$ maps $C_{r_{t}}$ onto the circular arc $L_t$. The preimages under $f_{t}$ of the two endpoints of the circular arc $L_t$ then partition $C_{r_{t}}$ into two circular arcs which are symmetric under the transformation $z\mapsto \overline{z}$ according to Lemma \ref{lemma:sym}. We call these circular arcs $\Gamma^{+}_{t}$ and $\Gamma^{-}_{t}$ respectively, where $\Gamma_{t}^{+}$ is the circular arc which intersects the negative real axis.

It can be shown that $\beta(t)\in\Gamma^{+}_{t}$ for all $t \in [0,T]$ implies that $y_{t}$ is strictly increasing. Similarly, $\beta(t)\in\Gamma^{-}_{t}$ for all $t \in [0,T]$ implies that $y_{t}$ is strictly decreasing. 

It would thus be sufficient to show that if $|\widetilde{\gamma}(t)|>y$, then $\beta(t)\in\Gamma^{+}_{t}$ for all $t \in [0,T]$. However, in general, this may not be the case. The idea of our method is as follows: Since $|\widetilde{\gamma}(t)|>y$, we can extend the length of $L$ to a longer circular arc $L^{*}_0$ without ever intersecting $\widetilde{\gamma}$. As the curve $\widetilde{\gamma}$  grows from $L^{*}_0$, we simultaneously shrink $L^{*}_{0}$ to $L^{*}_t$ such that at time $T$, $L^{*}_T$ coincides with $L$. By choosing a suitable rate at which $L^{*}_t$ shrinks to $L$ from each end of the circular arc $L^{*}_{0}$, we will be able to show that the preimage of 0 at each $t$, $y^{*}_{t}$, is now strictly increasing. This will prove the desired result since $y_{T}^{*}=y_{T}$ (as $L^{*}_{T}=L$ ) and $y_{0}^{*}=y_{0}$. The second equality follows from the fact that extending the length of $L$ to get $L^{*}_0$ does not change the preimage of $0$ by Lemma \ref{Lemma1}. It is for this part of the argument that we will need to use the three-slit Loewner differential equation from Section \ref{subsec:Prelim-Muti}: the three slits will be the curve $\gamma$, the circular arc in the clockwise direction and the circular arc in the anticlockwise direction.

The rest of this section provides the formal construction of the above argument.

\subsection{Properties of $P(r,y,\theta)$ in the Loewner-Type Differential Equation}
\label{sect:main_cal}
In this subsection, we study properties of the differential equation in (\ref{LDE}).
The following lemma gives some properties of $P(r,y,\theta )$ that we will need.
\begin{lemma}
\label{lem:P}
For $y \in (r,1)$, we have
\begin{enumerate} 
\item $P(r,y,\theta)= P(r,y,2\pi-\theta)$ for $\theta \in [0,2\pi]$.
\item $P(r,y,\theta)$ is increasing in $\theta$ for $\theta \in [0,\pi]$.
\item $P(r,y,\theta)$ is decreasing in $\theta$ for $\theta \in [\pi,2\pi]$.
\end{enumerate}
\end{lemma}

\begin{proof}
As $y>0$, we can see that
\[
1 - \lim\limits_{N \to \infty} \sum\limits_{n=-N}^N \dfrac{r^{2n-1} y + e^{i (2\pi-\theta) }} { r^{2n-1} y - e^{i (2\pi -\theta )} }
= 
\overline{ 
1 - \lim\limits_{N \to \infty} \sum\limits_{n=-N}^N \dfrac{r^{2n-1} y + e^{i \theta }} { r^{2n-1} y - e^{i \theta } }
}.
\]
Taking real parts of the above equation proves part 1.

To prove part 2, we write $r=e^{-p}$, $p>0$ and define
\[
A (z; p) = 
\prod\limits_{k=1}^\infty 
\left( 1- e^{-2kp} \right)
 \left( 1- e^{-(2k-1)p+iz} \right)
  \left( 1- e^{-(2k-1)p-iz} \right)
\]
This function $A$ is related to the Jacobi theta function $\vartheta_4$ defined in Section 13.19 of \cite{Book:theta}. Direct calculations show that 
\begin{align*}
    P(r,y,\theta)
    &= 2 \mathrm{Im} \left( 
\frac{    A' (\theta + i \ln y ; p)}{    A (\theta + i \ln y ; p)}
    \right)
\end{align*}
where $A'(z;p) = \partial_z A(z;p)$. It then follows that 
\begin{align*}
    \partial_\theta P(r,y,\theta)
    &= 2 \mathrm{Im} \left( 
\frac{    A'' (\theta + i \ln y ; p) A (\theta + i \ln y ; p) -  \left( A' (\theta + i \ln y ; p) \right)^2 }{   \left( A (\theta + i \ln y ; p) \right)^2}
    \right).
\end{align*}
Define 
\[
G_1 (z;p) = 
\frac{    A'' (z ; p) A (z ; p) -  \left( A' (z  ; p) \right)^2 }{   \left( A (z ; p) \right)^2}.
\]
Then $G_1 (z;p)$ is an elliptic function (of $z$) with periods $2\pi$ and $2ip$. Note that $G_1$ has poles of order $2$ at $z=2n\pi + (2m-1)ip$ for any $n,m\in \mathbb{Z}$ and these are the only poles of $G_1$. Let $\wp$ be the Weierstrass's $\wp$ function with period $2\pi$ and $2ip$. It follows that there exists an constant $c_1$ such that $G_2 (z):=G_1 (z;p)-c_1 \wp (z + ip)$ has no pole on $\mathbb{C}$. Thus $G_2 (z) = c_2$ for some constant $c_2$ by the Liouville's theorem. This means that we have 
\[
G_1 (z;p) = c_1 \wp (z + ip) + c_2
\]
for some constant $c_1 ,c_2$. By considering the Laurent series expansions of $G_1$ and $\wp$ and comparing coeffients, we have $c_1 =-1$ and $c_2$ is real, that is, 
\[
     G_1 ( z ; p ) = - \wp \left( z+ ip \right) + c 
\]
for some real constant $c$. (For details, see the demonstration by Dixit and Solynin \cite{Paper:kernel} for equation (3.3) of \cite{Paper:kernel}.)
Hence, 
\begin{align*}
    \partial_\theta P(r,y,\theta)
    &= - 2\mathrm{Im} \left( 
\wp \left( z +ip \right)
    \right).
\end{align*}
Note that $\wp $ maps the interior of the rectangle $R$ with vertices $z_1= 0$, $z_2= -ip$, $z_3= \pi - ip$ and  $z_4= \pi$ conformally into the lower half plane $\mathcal{H}=\lbrace z \in \mathbb{C} \: : \: \mathrm{Im} (z)<0 \rbrace$ and maps $\partial R$ injectively onto the real line (See for example, Section 13.25 of \cite{Book:theta}). Then, for any fixed $p = - \ln r$, for any $\theta \in (0, \pi)$ and for any $y \in (r,1)$, we have $z= \theta + i \ln y$ lies inside the interior of $R$ and hence $\partial_\theta P(r,y,\theta) >0$. Also, when $\theta=0$ or $\theta= \pi$, $z= \theta + i \ln y$ lies on $\partial R$ and hence $\partial_\theta P(r,y,\theta) =0$. 
This proves part 2. Part 3 follows from part 1 and part 2.
\end{proof}
As a consequence of Lemma \ref{lem:P}, we have the following lemma. 
\begin{lemma}
\label{lem:theta}
Let $y \in (r,1)$. Suppose $\theta_1,\theta_2 \in [0,2\pi)$ satisfy $|\pi-\theta_1| \leq |\pi-\theta_2|$. Then we have 
\[ P(r,y,\theta_1) \geq P(r,y,\theta_2). \]
\end{lemma}
\begin{proof}
When $\pi\leq\theta_{1}\leq\theta_{2}\leq2\pi$ or $0\leq\theta_{2}\leq\theta_{1}\leq \pi$, this result is a direct consequence of parts 2 and 3 of Lemma \ref{lem:P}. The remaining cases reduce to the above using part 1 in Lemma \ref{lem:P}.
\end{proof}
\subsection{The key result}
\label{subsec:proof:yt}
Let $\Omega$ be a circularly slit disk $\mathbb{D} \setminus L$ where $L$ is a circular arc centered at $0$ with radius $y_0$ for some $y_0 \in (0,1)$. Let $f$ be a conformal map of $A_{r_0}$ onto $\Omega$ such that $\left| f(z) \right| \to y_0$ as $\left| z \right| \to r_0$.
By Lemma \ref{Lemma1}, we have $|f^{-1}(0)|=y_0$. By composing $f$ with a rotation if necessary, we can assume without loss of generality that $f^{-1}(0)=y_0$. 
Suppose that $\widetilde{\gamma} : [0,T] \to \mathbb{D}$ is a Jordan arc satisfying $\widetilde{\gamma} (0) \in L$ and $\widetilde{\gamma} (0,T] \subset \Omega$, such that $\Omega \setminus \widetilde{\gamma} (0,t]$ has modulus $-(\log r_0)-t$ and let $\gamma = f^{-1} \circ \widetilde{\gamma}$.

\begin{figure}
\centering
\includegraphics[width=12cm]{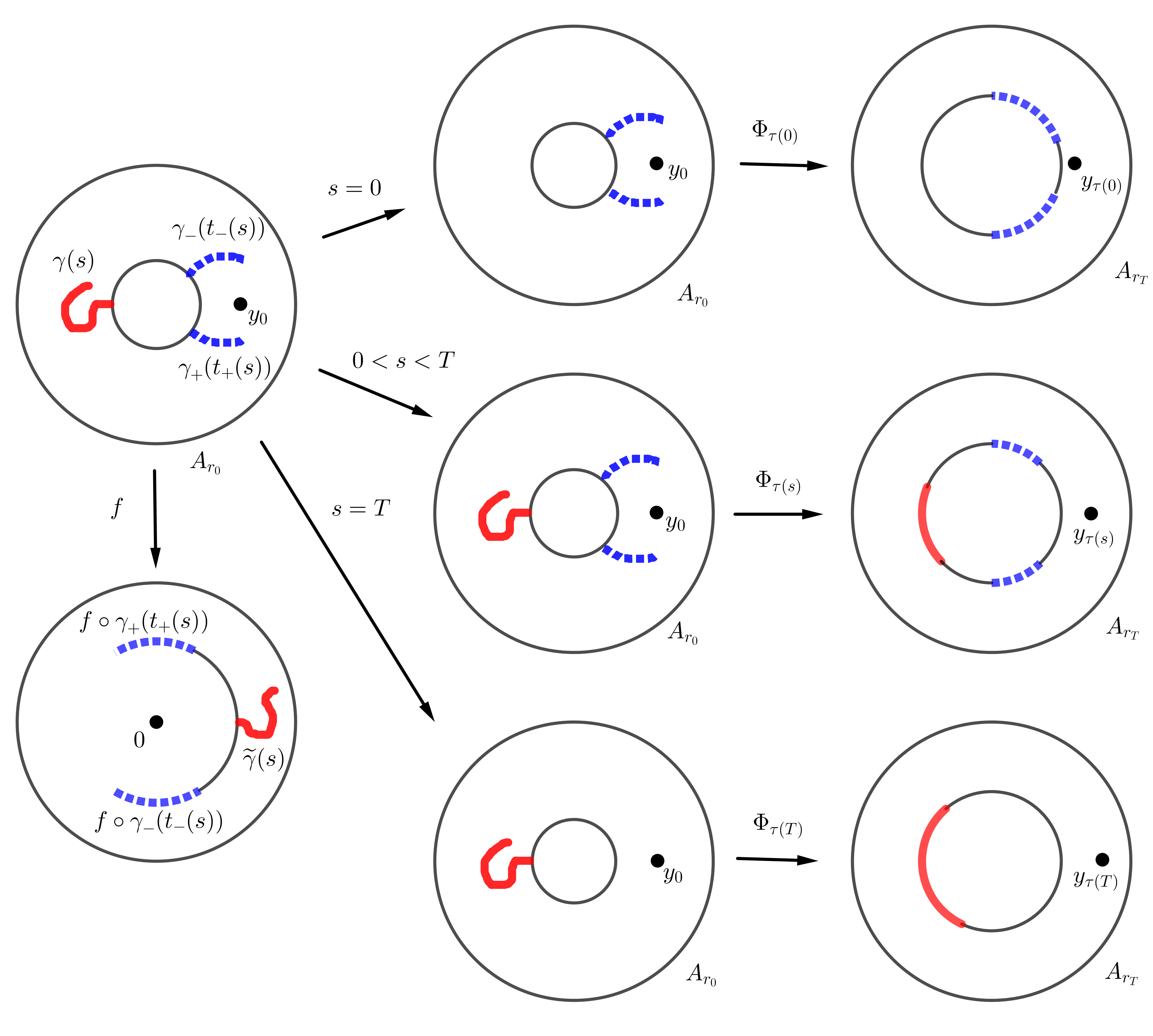}
\caption{Construction of $\Phi_{\tau}$.}
\label{fig:fig2}
\end{figure}

Let $\gamma_+: [0,T_+] \to \overline{A_{r_0}}$ be the Jordan arc such that $f\circ \gamma_+$ starts from an endpoint of the  circular arc $L$ and extends $L$ along the circular arc in the anticlockwise direction, i.e. $|f\circ \gamma_+ (t)|=y_0$ for all $t \in [0,T_+]$ and $f\circ \gamma_+ [0,T_+] \cap L = f\circ \gamma_+ (0)$. Similarly, let $\gamma_-: [0,T_-] \to \overline{A_{r_0}}$ be the Jordan arc such that $f\circ \gamma_-$ starts from an endpoint of the  circular arc $L$ and extends $L$ along the circular arc in the clockwise direction, i.e. $|f\circ \gamma_- (t)|=y_0$ for all $t \in [0,T_-]$ and $f\circ \gamma_-[0,T_-] \cap L = f\circ \gamma_- (0) \neq f\circ \gamma_+ (0)$. Moreover, $\gamma$, $\gamma_-$ and $\gamma_+$ are parametrized such that  $A_{r_0} \setminus \left(  \gamma(0,t_1]  \cup \gamma_+(0,t_2] \cup \gamma_-(0,t_3] \right)$ has modulus $p-|\tau|$, where $\tau= (t_1,t_2,t_3)$ and $|\tau|=t_1+t_2+t_3$.

Since, by assumption, $| \tilde{\gamma}(t)| > y_0$ for all $t \in (0,T]$, we have that $\widetilde{\gamma}$ does not intersect the circle of radius $y_{0}$. Hence $\gamma (0,T] \cap \gamma_-(0,T_-] =\emptyset$ and $\gamma (0,T] \cap \gamma_+ (0,T_+] =\emptyset$. We can also assume that $\gamma_- [0,T_-] \cap \gamma_+ [0,T_+] =\emptyset$. Define $r_{|\tau|}$ and the conformal maps 
\[ \Phi_{\tau} :  A_{r_0} \setminus \left( \gamma (0,t_1] \cup \gamma_+  (0,t_2] \cup \gamma_- (0,t_3] \right) \to A_{r_{|\tau|}} \]
as in Section \ref{subsec:Prelim-Muti} where the three curves are $\gamma,\gamma_{-},\gamma_{+}$. 
Let $y_{\tau}=\Phi_{\tau}(y_{0})$. In particular, for $\tau=(0,0,0)$
\[ y_{(0,0,0)} = \Phi_{(0,0,0)} (y_0)=  \mathrm{id} (y_0) = y_0.\]Also let $\beta(\tau)=\mathrm{Arg} \left( \Phi_{\tau}(\gamma(t_{1})) \right) \in [0,2\pi )$,  $\xi_{+}(\tau)=\mathrm{Arg} \left(  \Phi_{\tau}(\gamma_{+}(t_{2})) \right) \in [0,2\pi)$, $\xi_{-}(\tau)=\mathrm{Arg} \left(  \Phi_{\tau}(\gamma_{-}(t_{3})) \right) \in [0,2\pi)$. Now suppose that $a(s)$ is a real-valued differentiable function for $s\in[0,T]$ such that $\partial_s a(s)\in[0,1]$ for all $s\in[0,T]$. From now on, we assume that $\tau $ is a function of $s$ of the form, 
\begin{equation}
\tau (s) =(s, t_{+}(s) , \ t_{-}(s) )\label{tau}
\end{equation}
for $s \in[0,T]$ where $t_{+}(s)=(T-a(T))-(s-a(s))$ and $t_{-}(s)=a(T)-a(s)$ so that $|\tau (s)| \equiv T$. 
This construction is illustrated in Figure \ref{fig:fig2}.
Since $\partial_{s}a(s)\in[0,1]$, both $a(s)$ and $s-a(s)$ are non-decreasing and hence $t_{+}(s),t_{-}(s)$ are non-negative and non-increasing for $s\in [0,T]$. 

The function $a(s)$ affects the rate that  the circular arc \[L\cup \gamma_{+}[0,T_{+}]\cup \gamma_{-}[0,T_{-}]\] is shrinking from the clockwise end and the anticlockwise end. In the following lemma, we will choose a particular function $a(s)$ that will enable us to apply Lemma \ref{lem:theta}.   
\begin{lemma}
\label{lemma:Fixing}
There exists a real-valued differentiable function $a^{*}(s)$ with $0 \leq \partial_s a^*(s) \leq 1$ and $a^{*}(0) \geq 0$ such that, defining $\tau (s)$ as in equation (\ref{tau}) with $a(s)=a^{*}(s)$, we have 
\[  2\pi -\xi_+(\tau (s) )  = \xi_{-}(\tau (s) )  \]
for all $s \in [0,T] $.
\end{lemma}

\begin{proof}
Recall equation (\ref{MLDE}), 
\[ 
\partial_{s} \log \Phi_{\tau} (z)
=   \sum\limits_{i=1}^3  (\partial_s t_i) Q \left( r_{|\tau|},y_\tau, \xi_i (\tau) , \Phi_{\tau} (z) \right) .
\]
As $|\tau (s) |=T$ and noting that $s, t_{+}(s),t_{-}(s)$ are real-valued, taking imaginary parts on both sides of this equation yields
\begin{equation*}
\partial_{s}  \mathrm{Arg} \left(  \Phi_{\tau (s) } (z) \right)
=  H \left( r_{T},y_{\tau (s) },  \Phi_{\tau (s) } (z), \theta (s) , \partial_s a (s)  \right) 
\end{equation*}
where $\theta(s)= \left(\beta(\tau (s) ),\xi_{+}( \tau (s)  ),\xi_{-}( \tau (s)  ) \right)$ and
\[ H(r,y,w,\theta,\lambda) =  I(r,y,w,\theta_{1})  - (1-\lambda) I(r,y,w,\theta_2) -  \lambda I(r,y,w,\theta_3) \]
for $\theta=(\theta_{1},\theta_{2},\theta_{3})$.

First of all, notice that when $s=0$, 
\[ \Omega_0 = \mathbb{D} \setminus \left( L \cup \left( f \circ \gamma_{+}  (0,t_{+}(0)] \right) \cup \left( f \circ\gamma_{-} (0,t_{-}(0) ] \right)  \right) \]
is a circularly slit domain and $r_T e^{i \xi_{+}(\tau(0))}$ and $r_T e^{i \xi_{-}(\tau(0))}$ will be mapped to the end points $ f  ( \gamma_{+}  (t_{+}(0) ))$, $f (\gamma_{-} ( t_{-}(0) ) )$ of the circular slit under the conformal map $f \circ \Phi^{-1}_{\tau(0)} $. By applying a rotation to $\Omega_{0}$, Lemma \ref{lemma:sym} implies that
\begin{equation*}
    \label{initialcondition}
    \xi_{+}(\tau(0))= 2\pi- \xi_-(\tau (0)).
\end{equation*}

Suppose that $\epsilon>0$ is sufficiently small and $s\in(0,T-\epsilon]$. 
Note that both $\gamma_{+}(t_{+}( s+ \epsilon ))$  and $\gamma_{-}(t_{-}(s+\epsilon )$ have two preimages under $\Phi_{\tau (s) }$,
We define $u (s) $ and $v (s)$ to be the preimage under $\Phi_{\tau (s)}^{-1}$ of $\gamma_{+}(t_{+}( s+ \epsilon ))$  and $\gamma_{-}(t_{-}(s+\epsilon )$ respectively such that 
\[ \mathrm{Arg} (u(s)),\mathrm{Arg} (v (s))\in (\xi_{-}(\tau (s)),\xi_{+}(\tau (s))) \subset [0,2\pi). \]
Again, by applying a rotation to $\Omega_{0}$, Lemma \ref{lemma:sym} implies that \begin{equation}
\label{initialcondition2}
     \mathrm{Arg} (u(0))  =   2\pi - \mathrm{Arg} (v (0)) . 
\end{equation}
Then by the chain rule,
\begin{align*}
  & \partial_{s} \mathrm{Arg} \left(  u(s) \right) \\
 = & H(r_{T},y_{\tau (s)},u(s),\theta(s), \partial_s a(s) )+ \mathrm{Im}\left[ \frac{\Phi_{\tau (s)}'(\gamma_{+}(t_{+}(s+\epsilon ) ))}{\Phi_{\tau (s)}(\gamma_{+}(t_{+}(s+\epsilon )))}\partial_{s} [\gamma_{+}(t_{+}(s+\epsilon))] \right] 
 \\
  & \partial_{s} \mathrm{Arg} \left(  v(s) \right) \\
 = & H(r_{T},y_{\tau (s)},v(s),\theta(s), \partial_s a(s) )+ \mathrm{Im}\left[ \frac{\Phi_{\tau (s)}'(\gamma_{-}(t_{-}(s+\epsilon ) ))}{\Phi_{\tau (s)}(\gamma_{-}(t_{-}(s+\epsilon )))}\partial_{s} [\gamma_{-}(t_{-}(s+\epsilon))] \right] 
\end{align*}

Note that $\Phi_{\tau (s)} (z)$ is locally $2$ to $1$ at $\gamma_{+} ( t_{+} (s) )$ and $\gamma_{-} ( t_{-} (s) )$. 
Hence, $\Phi_{\tau (s)}'(\gamma_{+}(t_{+}(s)))=0$ and $\Phi_{\tau (s)}'(\gamma_{-}(t_{-}(s)))=0$. So for small enough $\epsilon$, $\Phi_{\tau (s)}'(\gamma_{+}(t_{+}(s+\epsilon )))$ and $\Phi_{\tau (s)}'(\gamma_{-}(t_{-}(s+\epsilon )))$ are bounded. 
Also, note that $\Phi_{\tau (s) } (z) \neq 0 $ near $\gamma_{-} (0, t_{-} (s)] \cup \gamma_{+} (0,t_{+} (s) ]$. Thus, 
\[ \mathrm{Im}\left[ \frac{\Phi_{\tau (s) }'(\gamma_{+}(t_{+}(s+\epsilon )))}{\Phi_{\tau (s)}(\gamma_{+}(t_{+}(s+\epsilon )))}\partial_{s} [\gamma_{+}(t_{+}(s+\epsilon ))] \right] 
\]
and 
\[
\mathrm{Im} \left[ \frac{\Phi_{\tau (s)}'(\gamma_{-}(t_{-}(s+\epsilon)))}{\Phi_{\tau (s)}(\gamma_{-}(t_{-}(s+\epsilon )))} \partial_{s}[\gamma_{-}(t_{-}(s+\epsilon ))] \right]
\]
are  bounded. 

Also, note that $H(r,y,w,\theta,\lambda)$ is continuous with respect to each variable for $w \neq re^{i\theta_1},re^{i\theta_2},re^{i\theta_3}$.

When $\lambda=0$, $H(r,y,w,\theta,0)$ is bounded near $w= re^{i\theta_3}$, and has a simple pole at $w=re^{i\theta_2}$. The pole at $w=re^{i\theta_2}$ arises from the expression $\mathrm{Im} \left( \frac{w+re^{i\theta_2}}{w-re^{i\theta_2}} \right)$ coming from the term $-I(r,y,w,\theta_{2})$ in the definition of $H(r,y,w,\theta,\lambda)$. In particular, the pole at $w=re^{i\theta_2}$ has residue $2$. Hence, for any given $M>0$, we can find some $\epsilon>0$ such that $H(r_{T},y_{\tau (s)},u(s),\theta(s),0) < -M$ since $\mathrm{Arg} (u(s)) \in (\xi_{-}(\tau (s)),\xi_{+}(\tau (s)))$ so that $\mathrm{Arg} (u(s))$ approaches the pole at $\mathrm{Arg}(w)=\xi_{+}(\tau(s))$ from the left. This implies that, when $a(s)=0$, 
\[ \partial_{s} \mathrm{Arg} \left( u(s) \right) +\partial_{s} \mathrm{Arg} \left( v(s) \right) \rightarrow -\infty \text{ as } \epsilon\rightarrow 0\]

Similarly, when $\lambda=1$, $H(r,y,w,\theta,1)$ is bounded near $w= re^{i\theta_2}$, and has a simple pole at $w=re^{i\theta_3}$. 
Again, the pole at $w=re^{i\theta_3}$ has residue $2$. Hence, for any given $M>0$, we can find some $\epsilon>0$ such that $H(r_{T},y_{\tau (s)},v(s),\theta(s),0) >M$ since $\mathrm{Arg} (v(s)) \in (\xi_{-}(\tau (s)),\xi_{+}(\tau (s)))$ so that $\mathrm{Arg} (v(s))$ approaches the pole at $\mathrm{Arg}(w)=\xi_{-}( \tau (s)  )$ from the right.  This implies that, when $a(s)=1$,
\[ \partial_{s} \mathrm{Arg} \left( u(s) \right) +\partial_{s} \mathrm{Arg} \left( v(s) \right) \rightarrow \infty \text{ as } \epsilon\rightarrow 0\]

Consequently, the intermediate value theorem implies that, for each $s\in[0,T]$, we can find $\lambda_{\epsilon}(s)\in [0,1]$ such that
\[ \partial_{s} \mathrm{Arg} \left( u(s) \right) +\partial_{s} \mathrm{Arg} \left( v(s) \right) =0\]
Hence, with 
\[a(s)=\int_{0}^{s}\lambda_{\epsilon}(s) ds,\] 
and using equation (\ref{initialcondition2}), we have
\begin{equation}
\label{eq:fixed}
\mathrm{Arg} \left( u(s) \right)=2\pi-\mathrm{Arg} \left( v(s) \right) \text{ for all } s.
\end{equation}

Let $\lambda^{*}(s) $ be the pointwise limit of $\lambda_\epsilon (s)$ as $\epsilon \to 0$. For all $s\in [0,T]$, $0 \leq \lambda_{\epsilon}(s) \leq 1$ and hence $0\leq \lambda^* (s)\leq 1$. Moreover, $\lambda^* (s)$ is integrable by the dominated convergence theorem. Define 
\[ a^* (s) = \int_0^s \lambda^* (s)  ds. \] 
Then with $a(s)=a^{*}(s)$ and using equation (\ref{eq:fixed}), we have
\[ \xi_{+}(\tau (s))=2\pi-\xi_{-}(\tau (s)) \text{ for all } s.\]
\end{proof}

As a consequence of Lemma \ref{lemma:Fixing}, we obtain the following key result.
\begin{proposition}
\label{thm:key}
If $| \tilde{\gamma}(t)| > y_0$ for all $t \in (0,T]$, we have $y_T > y_0$.
\end{proposition}

\begin{proof}
Let $a(s)=a^*(s)$ where $a^{*}(s)$ is given in Lemma \ref{lemma:Fixing}.  When $s=0$, we have 
\[ 
%y_{(0,T-a^*(T)+a^*(0),a^*(T)-a^*(0))}
y_{\tau (0)}=y_0
\] by Lemma \ref{Lemma1} as $L \cup f\circ\gamma_{+} [0,T-a(T)+a(0)] \cup f\circ\gamma_{-} [0,a(T)-a(0)]$ is  a circular arc. When $s=T$, we have $y_{(T,0,0)}=y_T$.  So it suffices to show that $\partial_{s}\log y_{\tau (s)} >0$.  By construction, we have
\[  2\pi - \xi_{-} (\tau (s)) =   \xi_{+} (\tau (s))  \]
for all $s \in [0,T] $. Note that we have $|\pi -\beta (\tau (0))| < |\pi -  \xi_{-} (\tau (0))|=|\pi -  \xi_{+} (\tau (0))|$. It follows that $|\pi -\beta (\tau (s))| < |\pi -  \xi_{-} (\tau (s)) |=|\pi -  \xi_{+} (\tau (s))|$ for all $s \in [0,T]$. Now note that equation (\ref{MLDE2}) can be rewritten as
\begin{align*}
\partial_{s}\log y_{\tau (s)} 
=&
\left(1-\partial_s a \right)
\left( P(r_{T},y_{\tau (s)}, \beta (\tau (s)) )-   P(r_{T},y_{\tau (s)},  \xi_{+} (\tau (s))) \right) 
\\ 
&
+\partial_s a \left(  P(r_{T},y_{\tau (s)}, \beta (\tau (s)) )-   P(r_{T},y_{\tau (s)},  \xi_{-} (\tau (s))) \right)
\end{align*}
As $0 \leq \partial_s a(s) \leq 1$ for all $s\in [0,T]$, Lemma \ref{lem:theta} implies that
\[\partial_{s}\log y_{\tau (s)} >0\]
Then the result follows.
\end{proof}
The above proposition allows us to prove Theorem \ref{MainResult3}. 

\begin{proof}[Proof of Theorem \ref{MainResult3}.] 
Define $E_{\min}= \left\lbrace z \in E : |z| \leq |w| \mbox{ for any $w\in E$} \right\rbrace $. So $E_{\min}$ consists of all the points in $E$ closest to the origin. Thus, $E_{min}$ is the union of circular arc(s) with the same radius $y_{0} \geq y$ and clearly $E_{min} \subset \partial E$. Note that either $E_{min} =E$ or $E_{min} \subsetneq E$. 

If $E_{min}=E$, then the connected set $E$ is a circular arc centered at $0$. In this situation, Lemma \ref{Lemma1} implies that $|g^{-1}(0)|=y_{0}$. 

If $E_{min} \subsetneq E$, then there are two subcases: either there exists a connected component $L$ of $E_{min}$ containing more than one point or $E_{min}$ is a set of disconnected points. 

Suppose that there is a connected component $L$ of $E_{min}$ containing more than one point. We first assume that $E \setminus L$ is a Jordan arc $\widetilde{\gamma}$. Then we can find $0<r_0<1$ such that $A_{r_0}$ has the same modulus as the domain $\mathbb{D} \setminus L$. Let $y_0>r_0$ and let $f(\cdot,y_0)$ be the conformal map of $A_{r_0}$ onto $\mathbb{D} \setminus L$ with $f(y_0,y_0)=0$. Define $\gamma : [0,T] \to \overline{A_{r_0}}$ to be the Jordan arc such that $\gamma (0) \in C_r$ and the image of $\gamma(0,T]$ under $f(\cdot,y_0)$ is $\widetilde{\gamma}$. In addition, $\gamma$ is parametrized such that $A_{r_0} \setminus \gamma (0,t]$ has modulus $-(\log r_0)+t$. We define $y_t$ as in Section \ref{sect:prelim:subsect:LDE}, namely $y_t=\Phi_t (y_0) $ where $\Phi_t$ is a conformal map from $A_{r_0} \setminus \gamma (0,t]$ onto $A_{r_t}$. Note that $\Phi_T=\Phi_{\tau (T)}$ and hence $y_T=y_{\tau (T)}$. Then by Proposition \ref{thm:key}, $y_T > y_0$. Since $g^{-1}(0)=y_T$, we have $|g^{-1}(0)| > y_0$. The case where $E \setminus L$ is not a Jordan arc follows from part 2 of Proposition \ref{dense}.

The final case where $E_{min}$ is a set of disconnected points follows from the previous case by letting the arc length of $L$ shrink to $0$. 

In all cases, $|g^{-1}(0)| \geq y_{0}\geq y$.

\end{proof}

%==========================
% Main Result
%==========================

\section{Proof of the Main Result}
\label{sect:proof}

Recall that the squeezing function is defined to be
\[ S_{A_r} (z) = \sup\limits_{f \in \mathcal{F}_{A_r}(z)}
\left\lbrace 
a
\: : \:
\mathbb{D}_{a} \subset  f(\Omega ) \subset \mathbb{D}
\right\rbrace. \]
where
\[ \mathcal{F}_{A_{r}}(z)=\left\lbrace 
f : \mbox{$f$ is a conformal map from $A_r$ to $\mathbb{C}$ such that $f(z)=0$.} 
\right\rbrace. \]
To simplify notation, we write $S_{r}(z)=S_{A_{r}}(z)$ and $\mathcal{F}_{r}(z)=\mathcal{F}_{A_{r}}(z)$. We have the following corollary of Theorem \ref{MainResult3}.
\begin{corollary} \label{reduce2}
Let 
\[ \widetilde{\mathcal{F}}_{r}(z) = \left\lbrace f \in \mathcal{F}_{r}(z) : f(A_r) \subset \mathbb{D}, f(\partial \mathbb{D})=\partial \mathbb{D} \right\rbrace \]
and define 
\[ \widetilde{S}_{r}(z) = \sup\limits_{f \in \widetilde{\mathcal{F}}_{r}(z)}
\left\lbrace a
\: : \:
\mathbb{D}_{a} \subset  f(\Omega )
\right\rbrace. \]
Then 
\[ \widetilde{S}_{r}(z) = |z| . \]
\end{corollary}
\begin{proof}
The conformal map $f$ of $A_{r}$ onto a circularly slit disk with $z$ mapping to 0 is in $\widetilde{\mathcal{F}}_{r}(z)$. Lemma \ref{Lemma1} implies that $\widetilde{S}_r(z) \geq  |z|$. 

Now suppose that we can find $f^{*}\in \widetilde{F}_{r}(z)$ such that $\mathbb{D}_{a^{*}}\subset f^{*}(A_{r})$ for some $a^{*}>|z|$. Let $E$ be the bounded component of the complement of $f^{*}(A_{r})$ in $\mathbb{C}$. Then $|w|\geq a^{*}$ for all $w\in E$. Theorem \ref{MainResult3} implies that, $|(f^{*})^{-1}(0)|\geq a^{*}$ which is a contradiction since $(f^{*})^{-1}(0)=z$.
\end{proof}

It now remains to prove Theorem \ref{MainResult1}. For any bounded doubly-connected domain $\Omega$, $\partial\Omega$ has two connected components: we denote the component that separates $\Omega$ from $\infty$  (i.e. the outer boundary) by $\partial^{o}\Omega$; we denote the other component (i.e. the inner boundary) by $\partial^{i}\Omega$. We decompose the family $\mathcal{F}_{r}(z)$ into two disjoint subfamilies

\[\mathcal{F}^{1}_{r}(z)=\left\{ f\in \mathcal{F}_r(z): f(\partial\mathbb{D})=\partial^{o}f(A_{r})\right\}\]  and
\[\mathcal{F}^{2}_{r}(z)=\left\{ f\in \mathcal{F}_r(z): f(\partial\mathbb{D})=\partial^{i}f(A_{r})\right\}.\] 
$\mathcal{F}^{1}_{r}(z)$ consists of functions that map outer boundary to outer boundary; $\mathcal{F}^{2}_{r}(z)$ consists of functions that interchange inner and outer boundary. We will consider a squeezing function on each subfamily separately. Define
\[ S^{1}_{r} (z) = \sup\limits_{f \in \mathcal{F}^{1}_{r}(z)}
\left\lbrace 
a
\: : \:
\mathbb{D}_{a} \subset  f(A_{r}) \subset \mathbb{D}
\right\rbrace \]
and
\[ S^{2}_{r} (z) = \sup\limits_{f \in \mathcal{F}^{2}_{r}(z)}
\left\lbrace 
a
\: : \:
\mathbb{D}_{a} \subset  f(A_{r} ) \subset \mathbb{D}
\right\rbrace. \]
Then the squeezing function satisfies $S_{r}(z)=\max\{S^{1}_{r}(z),S^{2}_{r}(z)\}$.
\begin{lemma}\label{reduce1}
\[S^{1}_{r}(z)=S^{2}_{r}\left(\frac{r}{z}\right)\]
\end{lemma}
\begin{proof}
This follows from the fact that $f\in\mathcal{F}^{1}_{r}(z)$ if any only if $ f \circ \rho \in \mathcal{F}^{2}_{r}(\frac{r}{z})$ where $\rho(z)=\frac{r}{z}$.
\end{proof}

\begin{proof}[Proof of Theorem \ref{MainResult1}.]
By Corollary \ref{reduce2} and Lemma \ref{reduce1}, it is sufficient to prove that $S_{r}^{1}(z)=\widetilde{S}_{r}(z)$. First we note that $\widetilde{\mathcal{F}}_{r}(z)\subset \mathcal{F}_{r}^{1}(z)$ and hence $\widetilde{S}_{r}(z)\leq S^{1}_{r}(z)$.

Since $\mathcal{F}^{1}_{r}(z)$ is a normal family of holomorphic functions, it follows easily that we can replace the $\sup$ in the definition of $S^{1}_{r}(z)$ with $\max$. Let $f\in \mathcal{F}^{1}_{r}(z)$ be any function that attains this maximum with corresponding $a$ i.e.
\[\mathbb{D}_{a}\subset f(A_{r})\subset \mathbb{D} \text{ and } a=S^{1}_{r}(z).\]
We denote by $\Omega$ the simply-connected domain satisfying $0 \in \Omega$ and $\partial \Omega = \partial^{o}f(A_{r})$. Note that $\mathbb{D}_{a}\subset\Omega\subset\mathbb{D}$. By the Riemann mapping theorem, we can find a conformal map $g$ of $\mathbb{D}$ onto $\Omega$ such that $g(0)=0$. Let  $F=g^{-1}\circ f$. Then $F\in \widetilde{\mathcal{F}}_{r}(z)$ and  
\[
\mathbb{D}_{a}\subset f(A_{r})
\Rightarrow 
g^{-1}(\mathbb{D}_{a})\subset F(A_{r}).\]
In addition, $g$ maps the unit disk into itself and so by the Schwarz lemma,
$g( \mathbb{D}_{a} )\subset \mathbb{D}_{a}$. Combining this with the above, we deduce that
\[\mathbb{D}_{a}\subset g^{-1}(\mathbb{D}_{a})\subset F(A_{r}).\]
 Hence $a\leq\sup\{\rho: \mathbb{D}_{\rho}\subset F(A_{r})\}$ which implies that $S^{1}_{r}(z)\leq \widetilde{S}_{r}(z)$. Therefore $S_{r}^{1}(z)=\widetilde{S}_{r}(z)$ as required.
\end{proof}

%=====================================================
% Application
%=====================================================
\section{The squeezing function on product domains in $\mathbb{C}^{n}$}
\label{sect:several}

It remains to prove Theorem \ref{thm:several}.
\begin{proof}[Proof of Theorem \ref{thm:several}] 
The squeezing function is scale invariant and hence in the definition of the squeezing function $S_{\Omega_{i}}(z_{i})$,  we can restrict the family $\mathcal{F}_{\Omega_{i}}(z_i)$ to 
\[ \mathcal{F}^{b}_{\Omega_{i}} ( z_i )=\{f\in\mathcal{F}_{\Omega_{i}}(z_{i}):|f(w)|< 1 \text{ for all } w\in\Omega_{i} \}.\] 
Since $\mathcal{F}^{b}_{\Omega_{i}}(z_i)$ is a normal family, one can easily show that there is a function $f_i$ in $\mathcal{F}^{b}_{\Omega_{i}}(z_i)$ attaining the supremum in the definition of $S_{\Omega_i}(z_i)$. By scaling, we can assume also that $\sup\{ \left| f_{i}(w) \right| : w\in\Omega_{i}\}=1$. Consider $\lambda_i = S_{\Omega_i}^{-1}(z_i)$ and $g(w)=(g_1(w_1) , \cdots , g_n (w_n))$ where $g_i = \lambda_i f_i$. Since $g_i$ are holomorphic and injective for all $i$, $g(w)$ is a holomorphic embedding of $\Omega$ into $\mathbb{C}^n$. Also, $f_{i}(z_{i})=0$ for each $i$ and thus $g\in\mathcal{F}_{\Omega}(z)$. 
Moreover, since $f_i$ attains the supremum in $S_{\Omega_i}(z_i)$, 
we have  $\mathbb{D}_{\lambda_i^{-1}} \subset f_i (\Omega_i) \subset \mathbb{D}$ and hence $\mathbb{D} \subset g_i (\Omega_i) \subset \mathbb{D}_{\lambda_i}$. It follows that 
\[\mathbb{D}^{n}\subset g (\Omega) \subset \mathbb{D}_{\lambda_1}\times\cdots\times \mathbb{D}_{\lambda_{n}}.\]
However,
$B(0;1)\subset \mathbb{D}^{n}$ and $\mathbb{D}_{\lambda_1}\times\cdots\times \mathbb{D}_{\lambda_{n}}\subset B(0;\Lambda)$ where \[\Lambda = \sqrt{\lambda_{1}^{2}+\cdots+\lambda_{n}^{2}}.\]
Hence
\[B(0;1)\subset g (\Omega) \subset B(0;\Lambda)\]
and we deduce that
\[
S_{\Omega} (z) \geq \frac{1}{\Lambda}=\left( S_{\Omega_1}^{-2}(z_1) + \cdots + S_{\Omega_n}^{-2}(z_n) \right)^{-1/2}.
\]
\end{proof}

%================================================
% Conjecture on the cases of multiply connected domains
%================================================
\section{The squeezing function on multiply-connected domains}
\label{sect:conjecture}
We now discuss some future directions regarding the squeezing function on planar domains of higher connectivity. 

Let $\Omega \subset \mathbb{C}$ be a finitely connected domain with disjoint boundary components $\gamma_0 , \gamma_1, \cdots \gamma_n$. As a corollary of conformal equivalence for finitely connected regions \cite{conway2012functions}, every finitely connected domain with non-degenerate boundary components is conformally equivalent to a circular domain, (i.e., the unit disk with smaller disks removed). Thus we can assume that $\gamma_{0}$ is the unit circle and $\gamma_{1},\ldots, \gamma_{n}$ are circles contained inside the unit disk. In light of our results, for a fixed $z\in \Omega$, we propose that the function which attains the maximum in the extremal problem 
\[ 
\sup\limits_{f \in \mathcal{F}_{\Omega}(z)}
\left\lbrace 
\frac{a}{b} 
\: : \:
\mathbb{D}_a \subset  f(\Omega ) \subset \mathbb{D}_b
\right\rbrace.
\]
is given by the conformal map of $\Omega$ onto a circularly slit disk of the same connectivity (i.e. the unit disk with proper arcs of circles centered at $0$ removed).

 For $j=1,\ldots,n$, let $\mu_{j}$ be a M\"{o}bius transformation that interchanges $\gamma_{j}$ and the unit circle $\partial \mathbb{D}$ and let $\Omega_{j}=\mu_{j}(\Omega)$. We also let $\mu_{0}$ be the identity mapping and $\Omega_{0}=\Omega$. Then, for $j=0,\ldots,n$, let $f_{j}$ denote the conformal map of $\Omega_{j}$ onto a circularly slit disk with $f_{j}(\mu_{j}(z))=0$ and let 
$\mathrm{Rad}(\Omega_j)$ be the minimum of the radii of the circular arcs in $f_{j}(\Omega_{j})$. Note that $\mathrm{Rad}(\Omega_{j})$ does not depend on the choice of $\mu_{j}$ or $f_{j}$ (by 
Theorem 6.2 in \cite{conway2012functions}). We make the following conjecture regarding the squeezing function on $\Omega$.

\begin{flushleft}
\textbf{Conjecture \:}
\textit{ $S_{\Omega} (z) = \max \{ \mathrm{Rad}(\Omega_j) \: : \: j=0,1,\cdots, n\}$.}
\end{flushleft}
The Schottky-Klein prime function $\omega(\cdot,\cdot)$ can be defined on domains of connectivity $n$ in terms of the Schottky Group of $\Omega$ (see \cite{crowdy2011schottky}).
By \cite{crowdy2005schwarz}, the same expression in Theorem \ref{CrowdyConformal},
\[ f( \cdot , z ) = \dfrac{\omega( \cdot ,z )}{|z| \omega( \cdot ,\overline{z}^{-1})}
 \]
gives the formula for the conformal map of $\Omega$ onto a circularly slit disk mapping $z$ to $0$. In addition, an expression for $\mathrm{Rad}(\Omega)$ is also given in \cite{crowdy2005schwarz}.

Recently, B\"ohm and Lauf \cite{bohm} have obtained an expression for a version of the Loewner differential equation on $n$-connected circularly slit disks. Using Crowdy's version of the Schwarz-Christoffel formula for multiply-connected domains (given in \cite{crowdy2005schwarz}), it should be possible to express the Loewner differential equation in terms of the Schottky-Klein prime function. It is anticipated that the methods in our paper could then be used to prove this conjecture.

\bigskip
\paragraph{Acknowledgments:}
The first author was partially supported by the RGC grant 17306019. The second author was partially supported by a HKU studentship. 

\bibliographystyle{spmpsci}   
\bibliography{reference}

\label{pg:reference}

\end{document}